\documentclass[11pt]{amsart}
\usepackage[margin=1.1in]{geometry}
\usepackage{amssymb,mathtools}
\usepackage[colorlinks=true,linkcolor=blue,citecolor=blue]{hyperref}

\newtheorem{theorem}{Theorem}[section]
\newtheorem{proposition}[theorem]{Proposition}
\newtheorem{lemma}[theorem]{Lemma}

\theoremstyle{definition}
\newtheorem{definition}[theorem]{Definition}

\usepackage{xcolor}

\newcommand{\R}{\mathbb{R}}
\newcommand{\Om}{\Omega}
\newcommand{\Lou}{\mathcal{L}}
\DeclareMathOperator{\dist}{dist}

\title[Concavity of the Ornstein--Uhlenbeck torsion function]{Concavity of the Ornstein--Uhlenbeck torsion function\\
on centered ellipses and its failure on convex domains}

\author{Xuan Hien Nguyen}
\address{Department of Mathematics, Iowa State University, Ames, IA 50011, USA}
\email{xhnguyen@iastate.edu}
\thanks{X. H. Nguyen is partially supported by the Dio Lewis Holl Fund.}

\author{Alina Stancu}
\address{Department of Mathematics and Statistics, Concordia University, Montreal, QC, Canada}
\email{alina.stancu@concordia.ca}
\thanks{A. Stancu is partially supported by the NSERC Discovery grant 2023-03926.}

\date{\today}

\begin{document}

\begin{abstract}
Let $u$ be the torsion function of a bounded domain $\Omega$ in $\mathbb{R}^n$
for the Ornstein--Uhlenbeck operator, that is, the solution of
$-\Delta u + x\cdot\nabla u = 1$ in $\Omega$ with $u=0$ on the boundary. We prove that for every $n\ge2$ there are
bounded, centrally symmetric domains whose torsion function is not concave. These domains have boundary with a flat piece and can be smoothed so that the boundary is $C^\infty$, strictly
convex with principal curvatures bounded below by a positive constant, while keeping the torsion function not concave. In contrast, the torsion function for every ellipse
centered at the origin is concave, and
for every ball centered at the origin we prove the sharp bound $D^2u\leq-\frac{1}{n}I$.  The
ellipse result follows by deformation from the disk, using a boundary
flux estimate and third-order boundary identities to exclude a first
zero eigenvalue of the Hessian.

\end{abstract}

\maketitle

\section{Introduction}\label{sec:intro}

Let $\gamma$ be the standard Gaussian measure on $\R^n$,
$d\gamma=(2\pi)^{-n/2}e^{-|x|^2/2}\,dx$, and let
\[
\Lou u=\Delta u-x\cdot\nabla u
\]
be the Ornstein--Uhlenbeck operator, so that
$\int\nabla u\cdot\nabla v\,d\gamma=\int(-\Lou u)\,v\,d\gamma$ for smooth
$u$ and compactly supported $v$.  For a bounded open set
$\Omega\subset\R^n$ the torsion function of $\Omega$ is the solution
$u=u_\Omega$ of
\begin{equation}\label{eq:torsion}
-\Lou u=1\ \text{ in }\Omega,\qquad u=0\ \text{ on }\partial\Omega.
\end{equation}

For the Laplacian in place of $\Lou$, the torsion function of a convex
domain need not be concave (see the equilateral triangle in
\cite[\S3]{KM93}), but $\sqrt u$ is concave (Makar-Limanov \cite{ML71} for
$n=2$, Kennington \cite{Ken85} for all $n$), and Kosmodem'yanskii \cite{Kos87} provided a curvature
condition on $\partial\Omega$ sufficient for concavity of $u$ itself.
For $\Lou$, related work on the log-concavity of the first eigenfunction and
of solutions of elliptic and parabolic problems includes
\cite{CFLS24,CQSell,CQS26,Qin25}.  In the companion paper \cite{NgBM}, we show
that the Brunn-Minkowski-type inequality  with exponent $1/3$ for the Gaussian
torsional rigidity conjectured by Mar\'in Sola and Salerno \cite{MSS26}
fails. The present paper is about pointwise concavity of the torsion function. We prove that convexity and central symmetry
of $\Omega$ do not ensure that $u$ is concave, while $u$ is concave on
every ellipse centered at the origin.\\

We say that a bounded domain is uniformly convex of class $C^\infty$ if
its boundary is a $C^\infty$ hypersurface whose principal curvatures are
bounded below by a positive constant.

\begin{theorem}\label{thm:main-nonconcave}
For every $n\ge2$, there exist bounded, centrally symmetric domains
$\Omega\subset\R^n$, uniformly convex of class $C^\infty$, whose torsion
function is not concave.
\end{theorem}

The domains are first constructed with a flat boundary face. On such a
face the normal derivative of $u$ cannot be constant, and its variation
makes a $2\times2$ principal submatrix of $D^2u$ indefinite
(Lemma~\ref{lem:face}), so $D^2u$ has a positive eigenvalue at some
interior point.  Since the torsion function depends continuously on the
domain in $C^2_{\mathrm{loc}}$, the positive eigenvalue at that point
persists under smooth approximation of the domain
(Lemmas~\ref{lem:approx} and~\ref{lem:stability}). We do not know
whether $\sqrt u$ or $\log u$ is concave on every bounded convex domain.
Finite-difference computations on planar stadiums, the domains with a
flat face used in Section~\ref{sec:nonconcave} for $n=2$, suggest that
$\sqrt u$ is concave there.

In the positive direction, concavity holds on balls and ellipses
centered at the origin.

\begin{theorem}\label{thm:main-concave}
Let $n\ge1$.
\begin{enumerate}
\item[(a)] For every $R>0$, the torsion function $u$ of the ball
$B_R=B_R(0)\subset\R^n$ centered at the origin satisfies
$D^2u\leq -\frac1n I$ on $\overline{B_R}$.
\item[(b)] The torsion function of every ellipse
centered at the origin satisfies $D^2u<0$ on the closed ellipse. In
particular, it is concave.
\end{enumerate}
\end{theorem}

Theorem~\ref{thm:main-concave}(b) is proved by deforming the disk
into the ellipse.  A zero eigenvalue of $D^2u$ cannot appear first at an
interior point, by the maximum principle for the second derivatives of
$u$ (Lemma~\ref{lem:interior}).  At a first zero of $\det D^2u$ on the
boundary, the outward normal derivative of $\det D^2u$ is nonpositive,
while an upper barrier for the boundary flux $-u_\nu$ and the exact
third-order boundary identities of the ellipse show that it is
positive.  The scheme is the equality-case analysis of
the continuity method for convexity properties of elliptic problems
\cite{CF85,KL87}, carried out on the boundary rather than in the
interior.

Section~\ref{sec:prelim} collects notation and standard facts.
Section~\ref{sec:nonconcave} proves Theorem~\ref{thm:main-nonconcave}.
Section~\ref{sec:ballsec} proves part (a) of
Theorem~\ref{thm:main-concave}.
Sections~\ref{sec:criterionsec} and~\ref{sec:touching} prove part (b):
Section~\ref{sec:criterionsec} reduces concavity to a boundary
criterion and proves the flux estimate, and Section~\ref{sec:touching}
carries out the deformation and the equality-case analysis.

\section{Preliminaries}\label{sec:prelim}

\subsection{Regularity and comparison}\label{sec:regularity}

Let $\Omega\subset\R^n$ be a bounded convex domain. By \cite[Thm.~6.13 and Thm.~6.17]{GT}, the problem
\eqref{eq:torsion} has a unique solution
$u\in C^{\infty}(\Omega)\cap C(\overline\Omega)$. Since the coefficients are
real-analytic, $u$ is real-analytic in $\Omega$ \cite{Hopf32,MN57}. If
$\partial\Omega$ is $C^\infty$, then $u\in C^\infty(\overline\Omega)$
\cite[Thm.~6.19]{GT}. If $\partial\Omega$ contains a relatively open flat
piece $F$, then $u\in C^\infty(\Omega\cup F)$ \cite[Lemma~6.18]{GT}. The strong maximum principle gives $u>0$ in
$\Omega$ \cite[Thm.~3.5]{GT}.

The operator is symmetric with respect to $\gamma$:
$e^{-|x|^2/2}\Lou u=\operatorname{div}\bigl(e^{-|x|^2/2}\nabla u\bigr)$.
Hence a function $\varphi\in W^{1,2}_{\mathrm{loc}}$ solves $\Lou\varphi=0$
weakly if and only if it solves the divergence-form equation
$\operatorname{div}(e^{-|x|^2/2}\nabla\varphi)=0$ weakly, and such
$\varphi$ is $C^\infty$ \cite[Cor.~8.11]{GT} and then real-analytic
\cite{Hopf32,MN57}.

\subsection{Notation} For symmetric matrices $A$ and $B$, we write
$A\leq B$ if $B-A$ is positive semidefinite and $A<B$ if $B-A$ is
positive definite.  For $R>0$, $B_R=B_R(0)$ is the open ball of radius
$R$ centered at the origin.

For $n\ge1$, let $u_n$ be the radial solution of
\begin{equation}\label{eq:radialODE}
u_n''+\frac{n-1}{r}u_n'-ru_n'=-1,\qquad u_n(0)=0,\quad u_n'(0)=0 .
\end{equation}
The torsion function of $B_R\subset\R^n$ is $u_n(|x|)-u_n(R)$.
\subsection{Support functions and Minkowski addition}\label{sec:convex}

For a convex body $K\subset\R^n$, let $h_K(\xi)=\max_{z\in K}z\cdot\xi$ be
its support function as a function on the unit sphere $\mathbb{S}^{n-1}$.  Minkowski addition of convex bodies corresponds to addition of their support functions,
$h_{K+ L}=h_K+h_L$. The
Hausdorff distance satisfies
$d_H(K,L)=\|h_K-h_L\|_{L^\infty({\mathbb{S}}^{n-1})}$, so
$d_H(K+ M,L + M)=d_H(K,L)$. A convex body $K$ is of class $C^\infty_+$ if $\partial K$ is a
$C^\infty$ hypersurface with positive Gauss curvature. Then $h_K$, extended by homogeneity to the complement of the origin, $h_K(y)=|y|h_K(\frac{y}{|y|})$, is
$C^\infty$ on $\R^n\setminus\{0\}$ and the principal radii of curvature of
$\partial K$ at the point with outer normal $\xi\in {\mathbb{S}}^{n-1}$ are the
eigenvalues of the Hessian of $h_K$ restricted to $\xi^\perp$
\cite[\S2.5]{Sch14}. Every convex body is a Hausdorff limit of bodies of
class $C^\infty_+$, and if $K$ is centrally symmetric the approximants can
be chosen centrally symmetric \cite[\S3.4]{Sch14}, \cite[p.~438]{Sch84}.

\section{Non-concavity on uniformly convex domains}\label{sec:nonconcave}

The proof of Theorem~\ref{thm:main-nonconcave} has three ingredients.
A domain with a flat boundary face has a non-concave torsion function
(Lemma~\ref{lem:face}), such a domain has uniformly convex smooth
approximants (Lemma~\ref{lem:approx}), and the torsion function depends
continuously on the domain in $C^2_{\mathrm{loc}}$
(Lemma~\ref{lem:stability}).

\begin{lemma}\label{lem:face}
Let $\Omega\subset\R^n$, $n\ge2$, be a bounded convex domain whose boundary
contains a nonempty relatively open subset $F$ of a hyperplane, and let $u$
be its torsion function. Then there exists $q\in\Omega$ with
$\lambda_{\max}\bigl(D^2u(q)\bigr)>0$.
\end{lemma}

The idea is due to Keady and McNabb: for the Laplacian in the plane,
$\det D^2u$ equals minus the square of the mixed normal--tangential
derivative on a straight boundary segment, so it is negative wherever the
normal derivative of $u$ varies along the segment \cite[Thm.~9.6 and
\S11]{KM93}.

\begin{proof}
After a rotation about the origin, which leaves $\Lou$ invariant, we may
choose coordinates $x=(x',x_n)$ with $F\subset\{x_n=y_0\}$ and
$\Omega\subset\{x_n<y_0\}$, so that the hyperplane containing $F$ supports
the convex set $\Omega$. Shrinking $F$, we may assume that $F$ is connected. By
Section~\ref{sec:regularity}, $u\in C^\infty(\Omega\cup F)$ and $u$ is
real-analytic in $\Omega$.

On $F$, $u\equiv0$, so $\partial_iu=\partial_i\partial_ju=0$ on $F$ for
$i,j<n$.  For $p\in F$ and $i<n$, the restriction of $D^2u(p)$ to
$\operatorname{span}\{e_i,e_n\}$ has matrix
\[
\begin{pmatrix}0&u_{in}(p)\\ u_{in}(p)&u_{nn}(p)\end{pmatrix},
\]
whose determinant is $-u_{in}(p)^2$.

Suppose first that $u_{in}(p)\neq0$ for some $p\in F$ and some $i<n$. Then
the quadratic form of $D^2u(p)$ is indefinite on
$\operatorname{span}\{e_i,e_n\}$, so $\lambda_{\max}(D^2u)(p)>0$, and by
continuity of $D^2u$ on $\Omega\cup F$ the same holds at points
$q\in\Omega$ near $p$.

Suppose next that $u_{in}\equiv0$ on $F$ for all $i<n$. Then all
tangential derivatives of $\partial_nu$ vanish on the connected set $F$, so
$\partial_nu\equiv c_0$ on $F$ for a constant $c_0$. Let $\widetilde w$ solve
\[
\widetilde w''(s)-s\,\widetilde w'(s)=-1,\qquad
\widetilde w(y_0)=0,\quad\widetilde w'(y_0)=c_0,
\]
and set $v(x)=\widetilde w(x_n)$. Then $\Lou v=-1$. The function
$\varphi=u-v$ satisfies $\Lou\varphi=0$ in $\Omega$ and
$\varphi=0$, $\nabla\varphi=0$ on $F$. Fix $p\in F$ and a ball
$D$ centered at $p$ with $D\cap\partial\Omega\subset F$, so that
$D^-=D\cap\{x_n<y_0\}\subset\Omega$. Let $\widetilde\varphi=\varphi$ on
$D^-$ and $\widetilde\varphi=0$ on $D\cap\{x_n\ge y_0\}$. Since $\varphi$
and $\nabla\varphi$ vanish on $F$, $\widetilde\varphi\in C^1(D)$, and for
$\psi\in C_c^\infty(D)$, integration by parts on $D^-$ gives
\[
\int_D e^{-|x|^2/2}\nabla\widetilde\varphi\cdot\nabla\psi\,dx
=-\int_{D^-}e^{-|x|^2/2}(\Lou\varphi)\psi\,dx
+\int_{D\cap F}e^{-|x|^2/2}\,\partial_n\varphi\,\psi\,dx'=0 .
\]
Thus $\widetilde\varphi$ is a weak solution of $\Lou\widetilde\varphi=0$
in $D$, hence real-analytic in $D$ by Section~\ref{sec:regularity}. It
vanishes on the open set $D\cap\{x_n>y_0\}$, so it vanishes on $D$. Hence
$u=\widetilde w(x_n)$ on $D^-$, and by real-analyticity on the connected
set $\Omega$, $u=\widetilde w(x_n)$ on $\Omega$.

Choose $y_1<y_0$ such that the slice $S=\Omega\cap\{x_n=y_1\}$ is
nonempty, thus it is a bounded open subset of the hyperplane $\{x_n=y_1\}$, and
$u\equiv\widetilde w(y_1)$ is constant on $S$. Since $u>0$ in $\Omega$,
$\widetilde w(y_1)>0$. Let $z^*$ be a boundary point of $S$ relative to the
hyperplane. Then $z^*\in\partial\Omega$, so $u(z^*)=0$ by continuity of $u$
on $\overline\Omega$, while $u(z)=\widetilde w(y_1)>0$ for $z\in S$
arbitrarily close to $z^*$. This contradiction shows that the second case
does not occur.
\end{proof}

\begin{lemma}\label{lem:approx}
Let $n\ge2$, $L,R_c>0$, $Q=[-L,L]^{n-1}\times\{0\}\subset\R^n$, and
$C=Q+\overline{B_{R_c}}$.  Then $C$ is a centrally symmetric convex body
whose boundary contains the relatively open flat subsets
$(-L,L)^{n-1}\times\{\pm R_c\}$, and for every $\delta>0$ there is a
centrally symmetric convex body $C_\delta$ of class $C^\infty_+$ with
$d_H(C_\delta,C)\le\delta$ whose principal curvatures lie in
$[c_\delta,R_c^{-1}]$ for some $c_\delta>0$.
\end{lemma}

\begin{proof}
Convexity, central symmetry and the description of the faces follow from
the definition of $C$. The set $Q_\delta=Q+\frac\delta2\overline{B_1}$ is a
centrally symmetric convex body in $\R^n$ with nonempty interior and
$d_H(Q_\delta,Q)=\delta/2$. By Section~\ref{sec:convex}, there exists a
centrally symmetric convex body $K_\delta\subset\R^n$ of class
$C^\infty_+$ with $d_H(K_\delta,Q_\delta)\le\delta/2$, hence
$d_H(K_\delta,Q)\le\delta$. Set $C_\delta=K_\delta+\overline{B_{R_c}}$.
Then $C_\delta$ is convex and centrally symmetric, and
$d_H(C_\delta,C)=d_H(K_\delta,Q)\le\delta$.

Since $K_\delta$ is of class $C^\infty_+$, its support function is
$C^\infty$ on $\R^n\setminus\{0\}$, and its principal radii of curvature
are positive and bounded above by some $r_\delta<\infty$. Now
$h_{C_\delta}=h_{K_\delta}+R_c|\cdot|$, and the Hessian of $R_c|\xi|$
restricted to $\xi^\perp$ is $R_c\,\mathrm{Id}$ for $|\xi|=1$. Hence
$h_{C_\delta}$ is $C^\infty$ with tangential Hessian between
$R_c\,\mathrm{Id}$ and $(R_c+r_\delta)\,\mathrm{Id}$, so $C_\delta$ is
of class $C^\infty_+$ with principal radii of curvature in
$[R_c,R_c+r_\delta]$, that is, principal curvatures in
$[(R_c+r_\delta)^{-1},R_c^{-1}]$.
\end{proof}

\begin{lemma}\label{lem:stability}
Let $\Omega_j$, $\Omega\subset\R^n$ be bounded convex domains with
$d_H(\Omega_j,\Omega)\to0$, and let $u_j$, $u$ be their torsion functions.
Then $u_j\to u$ in $C^2_{\mathrm{loc}}(\Omega)$.
\end{lemma}

\begin{proof}
All domains lie in a fixed ball $B_\rho$. Let $U$ be the torsion function of
$B_\rho$. Domain monotonicity gives $0\le u_j\le\max U$ for all $j$.

Fix $j$ and $\zeta\in\partial\Omega_j$, and let $\nu$ be a unit vector with
$\Omega_j\subset\{z:(z-\zeta)\cdot\nu\le0\}$. Set $s(z)=-(z-\zeta)\cdot\nu$,
$k=\rho+1$, $A=e^{2k\rho}/k$ and $b=A(1-e^{-ks})$. On $\Omega_j$,
\[
-\Lou b=Ak^2e^{-ks}-Ake^{-ks}(x\cdot\nu)\ge Ake^{-ks}(k-\rho)
\ge Ake^{-2k\rho}=1,
\]
and $b\ge0$ on $\partial\Omega_j$, so $u_j\le b\le Aks$ by the maximum
principle. Given $z\in\Omega_j$, choose $\zeta\in\partial\Omega_j$ with
$|z-\zeta|=\dist(z,\partial\Omega_j)$ and $\nu=(\zeta-z)/|\zeta-z|$.
Since $B_{|z-\zeta|}(z)\subset\Omega_j$, every supporting hyperplane of
$\Omega_j$ at $\zeta$ is tangent to this ball, so $\nu$ is admissible
and $s(z)=|z-\zeta|$. Hence
\begin{equation}\label{eq:barrier}
u_j(z)\le C_\rho\dist(z,\partial\Omega_j),\qquad C_\rho=Ak,
\end{equation}
uniformly in $j$.

Let $K\Subset\Omega$. Since $d_H(\Omega_j,\Omega)\to0$, $K\subset\Omega_j$
for large $j$. Interior Schauder estimates \cite[Thm.~6.2]{GT} give
$\|u_j\|_{C^{2,\alpha}(K)}\le C(K)$, so a subsequence converges in
$C^2_{\mathrm{loc}}(\Omega)$ to some $v$ with $-\Lou v=1$ in $\Omega$. For
$z\in\Omega$, the inclusions $\overline\Omega\subset\overline{\Omega_j}+
\overline{B_\varepsilon}$ and $\overline{\Omega_j}\subset\overline\Omega
+\overline{B_\varepsilon}$ for $\varepsilon>d_H(\Omega_j,\Omega)$ give
$|\dist(z,\partial\Omega_j)-\dist(z,\partial\Omega)|\le\varepsilon$ for
large $j$, so $\dist(z,\partial\Omega_j)\to\dist(z,\partial\Omega)$, and
\eqref{eq:barrier} gives
$0\le v(z)\le C_\rho\dist(z,\partial\Omega)$, so $v$ extends continuously
by $0$ to $\partial\Omega$. By uniqueness of the solution of
\eqref{eq:torsion}, $v=u$. Since every subsequence has a further
subsequence converging to $u$, the whole sequence converges.
\end{proof}

\begin{proof}[Proof of Theorem \ref{thm:main-nonconcave}]
Let $C=Q+\overline{B_{R_c}}$ be as in Lemma~\ref{lem:approx}, with any
$L,R_c>0$, let $\Omega=\operatorname{int}C$, and let $u$ be its torsion
function. By Lemma~\ref{lem:face} applied to the face
$(-L,L)^{n-1}\times\{R_c\}$, there are $q\in\Omega$ and $\varepsilon>0$
with $\lambda_{\max}(D^2u)(q)\ge2\varepsilon$. Let $C_\delta$ be the
bodies of Lemma~\ref{lem:approx}, $\Omega_\delta=\operatorname{int}C_\delta$,
and $u_\delta$ the torsion functions of $\Omega_\delta$.
Since $q$ is an interior point of $C$ and
$d_H(C_\delta,C)\to0$, $q\in\Omega_\delta$ for small $\delta$,
and by Lemma~\ref{lem:stability}, $D^2u_\delta(q)\to D^2u(q)$. Hence
$\lambda_{\max}(D^2u_\delta)(q)\ge\varepsilon$ for all small $\delta$, and
$u_\delta$ is not concave. The domains $\Omega_\delta$ are the
domains claimed.
\end{proof}

For $n=2$, the body $C$ is the stadium $[-L,L]\times\{0\}+\overline{B_{R_c}}$
and the face is a flat side, on which $\det D^2u=-u_{12}^2$. For $n\ge3$,
the full determinant of $D^2u$ vanishes on the face, because the
tangential $(n-1)\times(n-1)$ block is zero, and the obstruction is the
indefiniteness of a $2\times2$ principal submatrix mixing one tangential
direction with the normal.
The dichotomy in the proof of Lemma~\ref{lem:face}, that on a flat boundary
piece the normal derivative either varies along the piece or is constant
and then forces $u$ to be one-dimensional, is the analytic continuation
argument of Henrot and Oudet for the second Dirichlet eigenfunction
\cite{HO03}, \cite[\S4.2]{Hen06}: if a minimizer of $\lambda_2$ among
convex plane domains of given area is $C^{1,1}$, then its boundary
contains segments, and on such a segment the normal derivative of the
eigenfunction cannot be constant. In both problems, the flat piece is a
place where the boundary flux is forced to vary, and here it is this
variation, $u_{t\nu}\neq0$, that produces the indefinite Hessian.

The curvature lower bound in Lemma~\ref{lem:approx} degenerates as
$\delta\to0$. However, curvature size alone does not explain the distinction between these examples and ellipses centered at the origin, whose Ornstein--Uhlenbeck torsion functions are concave by Theorem~\ref{thm:main-concave}.

\section{Balls}
\label{sec:ballsec}

\begin{theorem}\label{thm:ball}
Let $n\ge1$ and let $u=u_n$ solve \eqref{eq:radialODE}. Then for all
$r>0$,
\[
u'(r)<-\frac rn\qquad\text{and}\qquad u''(r)<-\frac1n .
\]
Consequently the torsion function of the ball $B_R\subset\R^n$ centered
at the origin satisfies $D^2u\leq-\frac1nI$ on $\overline{B_R}$ for
every $R>0$.
\end{theorem}

\begin{proof}
Multiplying \eqref{eq:radialODE} by the integrating factor
$r^{n-1}e^{-r^2/2}$ gives
$\bigl(r^{n-1}e^{-r^2/2}u'\bigr)'=-r^{n-1}e^{-r^2/2}$, and integrating
from $0$ to $r$, with $r^{n-1}u'(r)\to0$ at the origin,
\[
 u'(r)=-e^{r^2/2}\,r^{1-n}\int_0^r t^{n-1}e^{-t^2/2}\,dt .
\]
The substitution $t=rs$ gives
\begin{equation}\label{eq:uprimeint}
 u'(r)=-r\,A(r),
 \qquad
 A(r):=\int_0^1 s^{n-1}\exp\Bigl(\frac{r^2}{2}(1-s^2)\Bigr)\,ds .
\end{equation}
For $r>0$ the exponent is positive on $(0,1)$, so
\[
 A(r)>\int_0^1 s^{n-1}\,ds=\frac1n,
 \qquad
 A'(r)=r\int_0^1 s^{n-1}(1-s^2)\exp\Bigl(\frac{r^2}{2}(1-s^2)\Bigr)\,ds>0 .
\]
Hence, for $r>0$,
\[
 \frac{u'(r)}{r}=-A(r)<-\frac1n,
 \qquad
 u''(r)=-A(r)-rA'(r)<-\frac1n .
\]

The Hessian of the radial function $u(|x|)$ has the radial eigenvalue
$u''(r)$ and the tangential eigenvalue $u'(r)/r$, the latter repeated
$n-1$ times, and both are less than $-\frac1n$ for $r>0$.  At
the origin, $A(0)=\frac1n$ and $A'(0)=0$, so $u''(0)=-\frac1n$ and
$D^2u(0)=-\frac1nI$.
\end{proof}

The constant $\frac1n$ cannot be improved, since $D^2u(0)=-\frac1nI$.
Part (a) of Theorem~\ref{thm:main-concave} is Theorem~\ref{thm:ball}.

\section{Ellipses: the boundary Hessian and the flux estimate}\label{sec:criterionsec}

In this section, we reduce the concavity of $u$ to an inequality for a
single function on $\partial\Om$.  Interior degeneracies of $D^2u$ are
excluded by the maximum principle, because each second derivative
$u_{\xi\xi}$ satisfies $(\Lou-2)u_{\xi\xi}=0$
(Lemma~\ref{lem:interior}).  On the boundary, parametrized
by the outward normal angle $\theta$, the entries of $D^2u$ are
expressed in terms of the boundary flux $q=-u_\nu$ and the geometry of
the ellipse (Proposition~\ref{prop:boundary_hessian}), and concavity
becomes the flux inequality \eqref{eq:criterion} of
Theorem~\ref{thm:criterion}.  Section~\ref{sec:barrierssec} proves an
upper barrier for $q$ and the flux estimate
(Lemma~\ref{lem:fluxest}) on which the deformation argument of
Section~\ref{sec:touching} rests.

\subsection{The torsion problem on an ellipse}\label{sec:setup}

Throughout Sections~\ref{sec:criterionsec} and~\ref{sec:touching},
\[
 \Om=\left\{(x,y)\in\R^2:
 \frac{x^2}{a^2}+\frac{y^2}{b^2}<1\right\},
 \qquad a\geq b>0,
\]
and $u$ is the torsion function of $\Om$, the solution of
\eqref{eq:torsion}.  Since $\Lou$ is invariant under rotations about the
origin, every ellipse centered at the origin reduces to this case.  The
boundary is analytic, so $u\in C^\infty(\overline\Om)$ by
Section~\ref{sec:regularity}, and the maximum principle and the Hopf
lemma give
\[
 u>0\quad\text{in }\Om,
 \qquad u_\nu<0\quad\text{on }\partial\Om,
\]
where $\nu$ is the outward unit normal.  By uniqueness and the symmetry
of $\Lou$, $u$ is even in each coordinate:
\[
 u(-x,y)=u(x,y),\qquad u(x,-y)=u(x,y).
\]
Since $\Lou$ is not invariant under dilations, the normalization $ab=1$
cannot be imposed.

\subsection{Interior maximum principle for the Hessian}

For a function $f$ and a constant unit vector $\xi\in\mathbb S^1$,
\begin{equation}\label{eq:commutator}
 \partial_\xi(\Lou f)=\Lou(f_\xi)-f_\xi .
\end{equation}
Fix $\xi$. Applying \eqref{eq:commutator} to $u$, with $\Lou u=-1$, gives
$\Lou (u_{\xi})=u_{\xi}$, and applying it again to $u_\xi$ gives
\begin{equation}\label{eq:hessian_equation}
 (\Lou-2)u_{\xi\xi}=0\quad\text{in }\Om
 \qquad\text{for every constant unit vector }\xi .
\end{equation}

At a point $x$, $D^2u(x)\leq0$ if and only if $u_{\xi\xi}(x)\leq0$ for
every unit vector $\xi$, and $D^2u\leq0$ on $\overline\Om$ is equivalent
to the concavity of $u$.

\begin{lemma}[Boundary reduction]\label{lem:interior}
If $D^2u\leq0$ on $\partial\Om$, then $D^2u\leq0$ on $\overline\Om$ and
$D^2u<0$ in $\Om$.  In particular, $u$ is concave.
\end{lemma}

\begin{proof}
Fix a unit vector $\xi$.  By \eqref{eq:hessian_equation}, $w=u_{\xi\xi}$
satisfies $(\Lou-2)w=0$ in $\Om$, and $\Lou-2$ has zeroth-order coefficient
$-2<0$.  Since $w\leq0$ on $\partial\Om$, the weak maximum principle for
such operators gives $w\leq0$ on $\overline\Om$.  If $w=0$ at an
interior point, then $w$ attains the nonnegative maximum $0$ in $\Om$,
so $w\equiv0$ in $\Om$ by the strong maximum principle
\cite[Thm.~3.5]{GT}.  Then $u$ is affine on every chord of $\Om$ parallel
to $\xi$, and since $u$ vanishes at both endpoints of the chord, $u\equiv0$
on it, contradicting $u>0$ in $\Om$.  Hence $u_{\xi\xi}<0$ in $\Om$ for
every $\xi$.
\end{proof}

For the Laplacian in the plane, this propagation of concavity from the boundary goes back at least to Keady and McNabb, who use the harmonicity of $\log(\frac14-\det D^2u)$ \cite[Thm.~9.1 and the proof of Thm.~9.3]{KM93}.

\subsection{Geometry of the ellipse}

Parametrize $\partial\Om$ by the outward normal angle $\theta$.  Thus
\[
 \nu(\theta)=(\cos\theta,\sin\theta),
 \qquad T(\theta)=(-\sin\theta,\cos\theta),
 \qquad \nu'=T,\quad T'=-\nu .
\]
The support function of the ellipse is
\begin{equation}\label{eq:support}
 h(\theta)=\sqrt{a^2\cos^2\theta+b^2\sin^2\theta},
 \qquad
 h'(\theta)=\frac{(b^2-a^2)\sin\theta\cos\theta}{h(\theta)} .
\end{equation}
The boundary point with outward normal $\nu(\theta)$ is
\begin{equation}\label{eq:P}
 P(\theta)=h\,\nu+h'\,T
 =\left(\frac{a^2\cos\theta}{h(\theta)},\ \frac{b^2\sin\theta}{h(\theta)}\right).
\end{equation}
Differentiating \eqref{eq:P} and using $\nu'=T$, $T'=-\nu$,
\[
 P'(\theta)=(h+h'')\,T=\rho\,T ,
\]
so $\theta\mapsto P(\theta)$ traverses $\partial\Om$ counterclockwise with
speed $\rho$. Thus, the radius of curvature and the curvature at $P(\theta)$
are, respectively,
\begin{equation}\label{eq:curvature}
 \rho(\theta)=h+h''=\frac{a^2b^2}{h^3},
 \qquad
 \kappa(\theta)=\frac1\rho=\frac{h^3}{a^2b^2}.
\end{equation}

Let $s$ be positively oriented arc length on $\partial\Om$, so that
$ds=\rho\,d\theta$, and let $\theta(s)$ be the inverse of
$s(\theta)=\int_0^\theta\rho$.  We write $P(s)$, $\nu(s)$, $T(s)$, $h(s)$
and so on for $P(\theta(s))$, $\nu(\theta(s))$, etc.: the same boundary
quantities viewed as functions of arc length.  Then $\frac{dP}{ds}=T$ and
the Frenet formulas for the outward-normal convention read
\begin{equation}\label{eq:frenet}
 \frac{dT}{ds}=-\kappa\,\nu,
 \qquad
 \frac{d\nu}{ds}=\kappa\,T,
 \qquad
 \frac{d}{ds}=\kappa\,\frac{d}{d\theta}.
\end{equation}
Throughout, a prime on a function of $\theta$ means $d/d\theta$, and a
subscript $s$ means $d/ds$, so that $f_s=\kappa f'$.

\subsection{The boundary Hessian and the concavity criterion}

\begin{definition}  The boundary flux is given by 
\[
 q(\theta):=-u_\nu\bigl(P(\theta)\bigr)>0 .
\]
\end{definition}
Since $u=0$ on $\partial\Om$, differentiating $u(P(s))=0$ along arc length
and using $\frac{dP}{ds}=T$ gives $u_T=\nabla u\cdot T=0$, so
\begin{equation}\label{eq:gradu}
 \nabla u=-q\,\nu\quad\text{on }\partial\Om .
\end{equation}

Write $u_{TT}=D^2u(T,T)$, $u_{T\nu}=D^2u(T,\nu)$, $u_{\nu\nu}=D^2u(\nu,\nu)$
at $P(s)$.

\emph{Tangential entry.}  Differentiating $u(P(s))=0$ twice along arc
length, with $\frac{dP}{ds}=T$, \eqref{eq:frenet} and \eqref{eq:gradu},
\[
 0=\frac{d^2}{ds^2}u(P(s))
 =D^2u(T,T)+\nabla u\cdot\frac{dT}{ds}
 =u_{TT}+\kappa q ,
\]
hence
\begin{equation}\label{eq:utt}
 u_{TT}=-\kappa q<0 .
\end{equation}

\emph{Mixed entry.}  Differentiating $q(s)=-\nabla u(P(s))\cdot\nu(s)$
along arc length, with \eqref{eq:frenet} and $\nabla u\cdot T=0$,
\[
 q_s=-D^2u(T,\nu)-\kappa\,\nabla u\cdot T=-u_{T\nu},
\]
hence
\begin{equation}\label{eq:utn}
 u_{T\nu}=-q_s .
\end{equation}

\emph{Normal entry.}  Evaluate $\Lou u=-1$ at $P(\theta)$: $\Delta u=u_{TT}+u_{\nu\nu}$
and, by \eqref{eq:gradu} and $P\cdot\nu=h$, $x\cdot\nabla u=-qh$.  So
$u_{TT}+u_{\nu\nu}+hq=-1$, and by \eqref{eq:utt}
\begin{equation}\label{eq:unn}
 u_{\nu\nu}=-E,
 \qquad E:=1+(h-\kappa)\,q .
\end{equation}

\begin{proposition}[Boundary Hessian]\label{prop:boundary_hessian}
In the orthonormal frame $\{T,\nu\}$,
\begin{equation}\label{eq:matrix}
 D^2u\big|_{P(\theta)}=
 \begin{pmatrix}
  -\kappa q&-q_s\\[3pt]
  -q_s&-E
 \end{pmatrix},
\end{equation}
and
\begin{equation}\label{eq:determinant}
 F:=\det D^2u\big|_{P(\theta)}
 =\kappa qE-q_s^2 .
\end{equation}
\end{proposition}

\begin{theorem}[Boundary criterion]\label{thm:criterion}
The following are equivalent:
\begin{enumerate}
 \item[(i)] $D^2u\leq0$ at every point of $\partial\Om$.
 \item[(ii)] $F\geq0$ on $\partial\Om$, that is,
 \begin{equation}\label{eq:criterion}
  q_s^2\leq\kappa\,qE\qquad\text{on }\partial\Om .
 \end{equation}
\end{enumerate}
If these hold, then $u$ is concave on $\overline\Om$.
\end{theorem}

\begin{proof}
By \eqref{eq:utt} the first diagonal entry of \eqref{eq:matrix} is negative,
and a symmetric $2\times2$ matrix with a negative diagonal entry is negative
semidefinite if and only if its determinant is nonnegative.  Global
concavity follows from Lemma~\ref{lem:interior}.
\end{proof}

\subsection{The upper flux barrier and the flux estimate}\label{sec:barrierssec}

This subsection proves an upper barrier for the flux $q$ by comparison
with an explicit quadratic (Proposition~\ref{prop:barriers}) and the
flux estimate of Lemma~\ref{lem:fluxest}, which is used in the
deformation argument of Section~\ref{sec:touching}.

The vertices show why such a bound is needed.  Since $u$ is even in
each coordinate,
\begin{equation}\label{eq:qsym}
 q(-\theta)=q(\theta),
 \qquad q(\pi-\theta)=q(\theta),
 \qquad\text{hence}\qquad
 q_s=0\ \text{ at }\theta=0,\tfrac\pi2 ,
\end{equation}
so at the four vertices the boundary Hessian \eqref{eq:matrix} is
diagonal and the criterion \eqref{eq:criterion} reduces to $E\geq0$.
At the major-axis vertex, where $h=a$ and $\kappa=a/b^2$, this reads
$(a/b^2-a)\,q(0)\leq1$, an upper bound on the flux when $b<1$.

\begin{proposition}[Upper flux barrier]\label{prop:barriers}
Let $S:=\frac1{a^2}+\frac1{b^2}$.  On $\partial\Om$,
\begin{equation}\label{eq:upperbarrier}
 (S-1)\,h\,q\leq1 .
\end{equation}
\end{proposition}

\begin{proof}
Let $\mathcal Q(x,y)=x^2/a^2+y^2/b^2$, so that
$\Lou\mathcal Q=2S-2\mathcal Q$.  On $\partial\Om$, the outward
normal is parallel to $\nabla\mathcal Q$, so by \eqref{eq:P}
\begin{equation}\label{eq:gradQ}
 \partial_\nu\mathcal Q=|\nabla\mathcal Q(P(\theta))|
 =2\sqrt{\frac{x^2}{a^4}+\frac{y^2}{b^4}}=\frac{2}{h(\theta)} .
\end{equation}
Let $v=(S-1)u-\tfrac12(1-\mathcal Q)$.  Then
\[
 \Lou v=-(S-1)+S-\mathcal Q=1-\mathcal Q\geq0\ \text{ in }\Om,
 \qquad v=0\ \text{ on }\partial\Om,
\]
so $v\leq0$ in $\Om$ by the maximum principle, and hence $v_\nu\geq0$
on $\partial\Om$.  By \eqref{eq:gradQ}, $v_\nu=-(S-1)q+1/h$, which
gives \eqref{eq:upperbarrier}.
\end{proof}

When $S\leq1$, inequality \eqref{eq:upperbarrier} holds trivially.  When
$S>1$, it is the upper bound $q\leq1/((S-1)h)$.  The flux estimate
below combines it with two identities for the ellipse.

\begin{lemma}\label{lem:Sident}
On $\partial\Om$,
\begin{equation}\label{eq:Sident}
 h'^2=\frac{(a^2-h^2)(h^2-b^2)}{h^2},
 \qquad
 \frac{\kappa}{h}+\frac{\kappa\,h'^2}{h^3}=S-\frac1{h^2}.
\end{equation}
\end{lemma}

\begin{proof}
From \eqref{eq:support}, $a^2-h^2=(a^2-b^2)\sin^2\theta$ and
$h^2-b^2=(a^2-b^2)\cos^2\theta$, so
\[
 h'^2=\frac{(a^2-b^2)^2\sin^2\theta\cos^2\theta}{h^2}
 =\frac{(a^2-h^2)(h^2-b^2)}{h^2},
\]
which also covers $a=b$.
Then, with $\kappa=h^3/(a^2b^2)$ from \eqref{eq:curvature},
\[
 \frac{\kappa}{h}\Bigl(1+\frac{h'^2}{h^2}\Bigr)
 =\frac{h^4+(a^2-h^2)(h^2-b^2)}{a^2b^2\,h^2}
 =\frac{(a^2+b^2)h^2-a^2b^2}{a^2b^2\,h^2}
 =S-\frac1{h^2}.
 \qedhere
\]
\end{proof}

\begin{lemma}[Flux estimate]\label{lem:fluxest}
On $\partial\Om$,
\begin{equation}\label{eq:fluxest}
 E-\kappa q\,\frac{h'^2}{h^2}\geq\frac qh>0 .
\end{equation}
In particular, $E>0$ on $\partial\Om$.
\end{lemma}

\begin{proof}
By \eqref{eq:unn} and \eqref{eq:Sident},
\[
 E-\kappa q\,\frac{h'^2}{h^2}
 =1+\Bigl(1-\frac{\kappa}{h}-\frac{\kappa\,h'^2}{h^3}\Bigr)hq
 =1+\Bigl(1+\frac1{h^2}-S\Bigr)hq
 =1-(S-1)hq+\frac qh,
\]
so the dependence on $\kappa$ cancels.  By \eqref{eq:upperbarrier},
$1-(S-1)hq\geq0$, which gives \eqref{eq:fluxest}.  Since
$\kappa q\,h'^2/h^2\geq0$, also $E\geq q/h>0$.
\end{proof}

At this stage, both diagonal entries of the boundary Hessian
\eqref{eq:matrix} are negative at every boundary point:
$u_{TT}=-\kappa q<0$, and $u_{\nu\nu}=-E<0$ by
Lemma~\ref{lem:fluxest}.  In particular, $D^2u<0$ at the four
vertices, where $q_s=0$ by \eqref{eq:qsym}.  What is missing is control of
the off-diagonal entry $-q_s$.  Section~\ref{sec:touching} obtains the missing inequality by
deformation, trading a pointwise bound on $q_s$ for the information
available at a first failure of \eqref{eq:criterion}.

\section{Deformation and the equality-case analysis}\label{sec:touching}

The idea is to deform the disk into $\Om$ through the family of ellipses $\Om_\tau$
with semi-axes $(\tau,b)$, $b\leq\tau\leq a$.  On the disk $\Om_b$,
$D^2u<0$ by Theorem~\ref{thm:ball}.  Suppose there is a first $\tau$ so that $D^2u_\tau\leq0$ on
$\partial\Om_\tau$ and that $F=\det D^2u_\tau$ vanishes at a boundary
point $p$.  By Lemma~\ref{lem:interior}, $F\geq0$ on
$\overline{\Om_\tau}$, so the outward normal derivative $F_\nu(p)$ is
nonpositive.  On the other hand, the derivative of $F$ along the
boundary vanishes at $p$, which determines $q_{ss}$, and the
third-order boundary identities of Lemma~\ref{lem:third} then give a
closed form for $F_\nu(p)$ (Lemma~\ref{lem:contact}), which the flux
estimate of Lemma~\ref{lem:fluxest} shows to be positive
(Proposition~\ref{prop:touching}).  This is a contradiction and shows that there is no possible first contact $F=0$ on $\partial \Omega_{\tau}$, see Proposition~\ref{prop:touching}. Finally, Theorem~\ref{thm:ellipse} gives the full deformation argument.

For a fixed ellipse, we suppress $\tau$.  The boundary identities of
Section~\ref{sec:criterionsec}, and those proved below, hold on every
ellipse.

\subsection{Third derivatives on the boundary}

At a boundary point write $u_{TT\nu}=D^3u(T,T,\nu)$,
$u_{T\nu\nu}=D^3u(T,\nu,\nu)$, $u_{\nu\nu\nu}=D^3u(\nu,\nu,\nu)$ for the
components of the third differential of $u$ in the frame $\{T,\nu\}$ at
that point.

\begin{lemma}[Third-order boundary identities]\label{lem:third}
On $\partial\Om$,
\begin{align}
 u_{TT\nu}&=-q_{ss}+\kappa\,(\kappa q-E),\label{eq:uTTn}\\
 u_{T\nu\nu}&=(\kappa_s-\kappa h')\,q-(h-3\kappa)\,q_s,\label{eq:uTnn3}\\
 u_{\nu\nu\nu}&=-q-h'q_s-hE-u_{TT\nu}.\label{eq:unnn3}
\end{align}
\end{lemma}

\begin{proof}
For vector fields $V,W$ along $\partial\Om$,
\[
 \frac{d}{ds}\,D^2u(V,W)=D^3u(T,V,W)+D^2u(V_s,W)+D^2u(V,W_s),
\]
and $T_s=-\kappa\nu$, $\nu_s=\kappa T$ by \eqref{eq:frenet}.

Differentiating \eqref{eq:utn}, with $V=T$, $W=\nu$, gives
\[
 u_{TT\nu}-\kappa u_{\nu\nu}+\kappa u_{TT}=-q_{ss},
\]
and inserting $u_{TT}=-\kappa q$ and $u_{\nu\nu}=-E$ from
\eqref{eq:utt} and \eqref{eq:unn} gives \eqref{eq:uTTn}.

Differentiating \eqref{eq:unn}, with $h_s=\kappa h'$ and $V=W=\nu$,
\[
 u_{T\nu\nu}+2\kappa u_{T\nu}=-E_s=-(\kappa h'-\kappa_s)q-(h-\kappa)q_s ,
\]
where inserting $u_{T\nu}=-q_s$ gives \eqref{eq:uTnn3}.

For \eqref{eq:unnn3}, fix a constant unit vector $e$ and differentiate
\eqref{eq:torsion} in the direction $e$:
\[
 \Delta u_e-x\cdot\nabla u_e-u_e=0\quad\text{in }\Om .
\]
At a boundary point, $\Delta u_e=D^3u(T,T,e)+D^3u(\nu,\nu,e)$ and
$x\cdot\nabla u_e=D^2u(h\nu+h'T,e)$ by \eqref{eq:P}.  Taking for $e$ the
value of $\nu$ at the point, with $u_\nu=-q$, $u_{T\nu}=-q_s$ and
$u_{\nu\nu}=-E$,
\[
 u_{TT\nu}+u_{\nu\nu\nu}=-h'q_s-hE-q ,
\]
which is \eqref{eq:unnn3}.
\end{proof}

\subsection{The contact identity and the no boundary contact proposition}\label{sec:touchconfig}

Extend $F$ to $\overline\Om$ by $F:=\det D^2u$. This
agrees with \eqref{eq:determinant} on $\partial\Om$.  At a boundary point, in the frame
$\{T,\nu\}$ held fixed, $F=u_{TT}u_{\nu\nu}-u_{T\nu}^2$, so by
\eqref{eq:matrix}
\begin{equation}\label{eq:Fnu}
 F_\nu=u_{TT\nu}\,u_{\nu\nu}+u_{TT}\,u_{\nu\nu\nu}-2u_{T\nu}\,u_{T\nu\nu}
 =-E\,u_{TT\nu}-\kappa q\,u_{\nu\nu\nu}+2q_s\,u_{T\nu\nu} .
\end{equation}

\begin{lemma}[Contact identity]\label{lem:contact}
Suppose that at the boundary point $p=P(\theta^*)$
\begin{equation}\label{eq:C1}
 F=0,\quad\text{that is}\quad q_s^2=\kappa qE,
 \qquad\text{and}\qquad F_s=0 .
\end{equation}
Then $q_s\neq0$ at $p$, and with
\begin{equation}\label{eq:beta}
 \beta:=\frac{h'q_s}{hE}
\end{equation}
we have
\begin{equation}\label{eq:G}
 F_\nu(p)=\frac\kappa2\Bigl[(1+hq)(3+2hq)(1+\beta)+2q^2\Bigr].
\end{equation}
\end{lemma}

We postpone the proof of Lemma~\ref{lem:contact} and first show how
it is used to prove that a zero eigenvalue of $D^2u$, which cannot occur
at an interior point by Lemma~\ref{lem:interior}, cannot occur on the
boundary either.

\begin{proposition}[No boundary contact]\label{prop:touching}
If $D^2u\leq0$ on $\partial\Om$, then $F>0$ on $\partial\Om$, that is,
$D^2u<0$ on $\partial\Om$.
\end{proposition}

\begin{proof}
By Lemma~\ref{lem:interior}, $D^2u\leq0$ on $\overline\Om$, so
$F=\det D^2u\geq0$ on $\overline\Om$.  Suppose that $F=0$ at some
$p\in\partial\Om$.  Since $F\geq0$ on $\partial\Om$, also $F_s(p)=0$,
so \eqref{eq:C1} holds and Lemma~\ref{lem:contact} applies.

By \eqref{eq:C1} and the flux estimate \eqref{eq:fluxest}, at $p$
\[
 \beta^2=\frac{h'^2\,\kappa qE}{h^2E^2}
 =\frac{\kappa q\,h'^2}{E\,h^2}
 =1-\frac1E\Bigl(E-\kappa q\,\frac{h'^2}{h^2}\Bigr)<1 ,
\]
so $1+\beta>0$ and \eqref{eq:G} gives $F_\nu(p)>0$.  On the other hand,
$F(p-t\nu)\geq0$ for small $t>0$ and $F(p)=0$, so $F_\nu(p)\leq0$.
This contradiction proves the claim.
\end{proof}

\begin{proof}[Proof of Lemma~\ref{lem:contact}]
All boundary quantities are evaluated at $\theta^*$. Since \(q>0\) by the Hopf lemma, \(\kappa>0\) by \eqref{eq:curvature}, and \(E>0\) by Lemma \(\ref{lem:fluxest}\), the contact relation \eqref{eq:C1} gives \(q_s^2=\kappa qE >0\). This  implies \(q_s\ne0\).  The definition \eqref{eq:beta} of $\beta$ and $q_s^2=\kappa qE$ give
\begin{equation}\label{eq:rbeta}
 \kappa q\,\frac{h'}{h}=\beta\,q_s .
\end{equation}
Below, $h'$ appears only in the combinations $h'q_s$ and $\kappa qh'/h$,
which \eqref{eq:beta} and \eqref{eq:rbeta} express through $\beta$.

\emph{The second derivative of the flux.}  By \eqref{eq:curvature},
$\kappa'=3h^2h'/(a^2b^2)=3\kappa h'/h$, so $\kappa_s=\kappa\kappa'=3\kappa^2h'/h$,
and differentiating $E=1+(h-\kappa)q$ along the boundary gives
$E_s=(\kappa h'-\kappa_s)q+(h-\kappa)q_s$.  Differentiating
$F=\kappa qE-q_s^2$ along the boundary,
\[
 F_s=\kappa_sqE+\kappa q_sE+\kappa qE_s-2q_sq_{ss}
 =\kappa q_s\bigl[E+(h-\kappa)q\bigr]
 +\kappa_sq\,(E-\kappa q)+\kappa^2h'q^2-2q_sq_{ss} .
\]
Here $E+(h-\kappa)q=2E-1$, and with $\kappa_s=3\kappa^2h'/h$ and
$hq=E+\kappa q-1$,
\[
 \kappa_sq\,(E-\kappa q)+\kappa^2h'q^2
 =\frac{\kappa^2h'q}{h}\bigl[3(E-\kappa q)+hq\bigr]
 =\frac{\kappa^2h'q}{h}\,(4E-2\kappa q-1)
 =\kappa\beta q_s\,(4E-2\kappa q-1)
\]
by \eqref{eq:rbeta}.  The condition $F_s=0$ therefore gives
\begin{equation}\label{eq:C2}
 q_{ss}=\frac\kappa2\bigl[2E-1+\beta\,(4E-2\kappa q-1)\bigr].
\end{equation}

\emph{Third derivatives at $p$.}  Inserting \eqref{eq:C2} into
\eqref{eq:uTTn},
\begin{align*}
 u_{TT\nu}=-q_{ss}+\kappa(\kappa q-E)
 &=-\frac\kappa2\bigl[(2E-1)-2(\kappa q-E)\bigr]
 -\frac\kappa2\,\beta\,(4E-2\kappa q-1)\\
 &=-\frac\kappa2(4E-2\kappa q-1)(1+\beta).
\end{align*}
In \eqref{eq:uTnn3}, $\kappa_s-\kappa h'=(3\kappa-h)\,\kappa h'/h$, so by
\eqref{eq:rbeta}
\[
 u_{T\nu\nu}=(3\kappa-h)\,\kappa q\,\frac{h'}{h}-(h-3\kappa)\,q_s
 =(3\kappa-h)\,q_s\,(1+\beta).
\]
By \eqref{eq:unnn3} and $h'q_s=hE\beta$,
\[
 u_{\nu\nu\nu}=-q-hE(1+\beta)-u_{TT\nu}.
\]

\emph{The normal derivative of $F$.}  By \eqref{eq:Fnu} and
$q_s^2=\kappa qE$,
\begin{align*}
 F_\nu
 &=-E\,u_{TT\nu}-\kappa q\bigl(-q-hE(1+\beta)-u_{TT\nu}\bigr)
 +2q_s\,u_{T\nu\nu}\\
 &=-(E-\kappa q)\,u_{TT\nu}+\kappa q^2+\kappa qhE(1+\beta)
 +2(3\kappa-h)\,\kappa qE\,(1+\beta)\\
 &=\frac\kappa2\Bigl[(E-\kappa q)(4E-2\kappa q-1)+2(6\kappa-h)\,qE\Bigr](1+\beta)
 +\kappa q^2 .
\end{align*}
With $hq=E+\kappa q-1$,
\begin{align*}
 (E-\kappa q)(4E-2\kappa q-1)+2(6\kappa-h)\,qE
 &=(E-\kappa q)(4E-2\kappa q-1)+12\kappa qE-2E(E+\kappa q-1)\\
 &=2(E+\kappa q)^2+(E+\kappa q)
 =(1+hq)(3+2hq).
\end{align*}
Inserting this last equality in the formula for $F_{\nu}$ gives \eqref{eq:G}.
\end{proof}

\subsection{Deformation}

\begin{theorem}[Concavity]\label{thm:ellipse}
The torsion function of every ellipse centered at the origin satisfies
$D^2u<0$ on $\overline\Om$, in particular it is concave.
\end{theorem}

\begin{proof}
Fix $b>0$ and $a\geq b$, and consider the family $\Om_\tau$ of ellipses
with semi-axes $(\tau,b)$, $\tau\in[b,a]$, with torsion functions
$u_\tau$, and let $F_\tau=\det D^2u_\tau$.

We first establish continuity in $\tau$. Define 
$A_\tau: B_1 \to \Omega_{\tau}$ by $ A_\tau z = (\tau z_1, b z_2)$ and let  $v_\tau(z)=u_\tau(A_\tau z)$.
The map $A_\tau$ transforms \eqref{eq:torsion} into
\[
 L_\tau v_\tau=-1\ \text{ in }B_1,
 \qquad v_\tau=0\ \text{ on }\partial B_1,
 \qquad
 L_\tau:=\tau^{-2}\partial^2_{z_1z_1}+b^{-2}\partial^2_{z_2z_2}-z\cdot\nabla .
\]
The operators $L_\tau$ are uniformly elliptic for $\tau\in[b,a]$, with
coefficients bounded in $C^{0,\alpha}(\overline{B_1})$ uniformly in
$\tau$.  The maximum principle \cite[Thm.~3.7]{GT} and the global
Schauder estimate on the fixed smooth domain $B_1$ \cite[Thm.~6.6]{GT}
give $\|v_\tau\|_{C^{2,\alpha}(\overline{B_1})}\leq C$ with $C$
independent of $\tau$.  For $\sigma,\tau\in[b,a]$, the difference
$w=v_\tau-v_\sigma$ satisfies
\[
 L_\tau w=(\sigma^{-2}-\tau^{-2})\,(v_\sigma)_{z_1z_1}\ \text{ in }B_1,
 \qquad w=0\ \text{ on }\partial B_1,
\]
and the right side has $C^{0,\alpha}(\overline{B_1})$ norm at most
$C|\tau-\sigma|$.  The same two estimates give
\[
 \|v_\tau-v_\sigma\|_{C^{2,\alpha}(\overline{B_1})}\leq C|\tau-\sigma| .
\]
Since $D^2u_\tau(A_\tau z)=A_\tau^{-1}D^2v_\tau(z)A_\tau^{-1}$,
\[
 F_\tau(A_\tau z)=\frac{\det D^2v_\tau(z)}{\tau^2b^2},
 \qquad
 \min_{\partial\Om_\tau}F_\tau=\frac{1}{\tau^2b^2}\,\min_{|z|=1}\det D^2v_\tau(z),
\]
and $\tau\mapsto\min_{\partial\Om_\tau}F_\tau$ is continuous on $[b,a]$.

Consider
\[
 \mathcal I=\Bigl\{\tau\in[b,a]:\ \min_{\partial\Om_\tau}F_\tau>0\Bigr\}.
\]

$\mathcal I$ is nonempty: $b\in\mathcal I$ because on the disk, $D^2u_b\leq-\tfrac12I$ by Theorem~\ref{thm:ball} so $F_b>0$. 

$\mathcal I$ is open in $[b,a]$ by continuity of $\tau\mapsto\min F_\tau$.

$\mathcal I$ is closed: let $\tau_j\in\mathcal I$, $\tau_j\to\tau^*$. By continuity,
$F_{\tau^*}\geq0$ on $\partial\Om_{\tau^*}$, so
$D^2u_{\tau^*}\leq0$ on $\partial\Om_{\tau^*}$ by
Theorem~\ref{thm:criterion}, and Proposition~\ref{prop:touching} gives
$F_{\tau^*}>0$ on $\partial\Om_{\tau^*}$. Hence $\tau^*\in\mathcal I$.

$\mathcal I$ is nonempty, open and closed in $[b,a]$, so $\mathcal I=[b,a]$. In
particular $a\in\mathcal I$. Then $D^2u_a<0$ on $\partial\Om_a$ by
Theorem~\ref{thm:criterion}, hence $D^2u_a<0$ on $\overline{\Om_a}$
by Lemma~\ref{lem:interior}, which is the claim.
\end{proof}

\begin{proof}[Proof of Theorem \ref{thm:main-concave}]
Part (a) is Theorem~\ref{thm:ball}, and part (b) is
Theorem~\ref{thm:ellipse}, using the rotation invariance noted in
Section~\ref{sec:setup}.
\end{proof}

\section{Research process}

The domains of Theorem~\ref{thm:main-nonconcave} and the estimates of
Sections~\ref{sec:criterionsec} and identities of Section~\ref{sec:touching} were found with the help of the language
model Claude Fable 5.0, through symbolic calculations, numerical searches
for counterexamples, and model-assisted exploration and review of
proofs. The deformation and equality-case proof of
Theorem~\ref{thm:main-concave} was inspired by the proofs of
log-concavity for first eigenfunctions \cite{SWYY85,LW87} and of the
constant rank theorems \cite{CF85,KL87}. The paper was cross-verified by GPT-6 Astra. None of these calculations or model outputs is used as
evidence in the proofs above. The authors have checked all proofs, edited the writing carefully, and
take full responsibility for the manuscript.


\begin{thebibliography}{CFLS26}
\bibitem[Kos87]{Kos87} A.~A.~Kosmodem'yanskii, Jr.,
\emph{Sufficient conditions for the concavity of the solution of the
Dirichlet problem for the equation $\Delta u=-1$},
Mat.\ Zametki \textbf{42} (1987), 537--542.
\end{thebibliography}
\end{document}